\documentclass[11pt]{article}
\usepackage[a4paper,margin=30mm]{geometry}
\usepackage{amsmath,amssymb,amsthm,mathtools}
\usepackage{microtype}
\usepackage{enumitem}
\usepackage{hyperref}
\usepackage[nameinlink,noabbrev]{cleveref}
\usepackage[numbers,sort&compress]{natbib}
\usepackage{booktabs}
\usepackage[T1]{fontenc}
\usepackage{lmodern}
\hypersetup{colorlinks=true,linkcolor=blue,citecolor=blue,urlcolor=blue}
\newtheorem{theorem}{Theorem}[section]
\newtheorem{proposition}[theorem]{Proposition}
\newtheorem{lemma}[theorem]{Lemma}
\newtheorem{corollary}[theorem]{Corollary}
\theoremstyle{remark}
\newtheorem{remark}[theorem]{Remark}

\newcommand{\R}{\mathbb R}
\newcommand{\E}{\mathbb E}
\newcommand{\Pp}{\mathbb P}
\newcommand{\1}{\mathbf 1}
\newcommand{\ns}{\mathrm{ns}}
\newcommand{\loc}{\mathrm{loc}}

\title{Initialization-Dependent Solution Selection by the Euler--Maruyama Scheme for a Degenerate SDE}
\author{Akihiro Tanaka\thanks{Sumitomo Mitsui Banking Corporation, 1-1-2 Marunouchi, Chiyoda-ku, Tokyo 100-0005, Japan. Email: \texttt{tnkaki2000@gmail.com}.}}
\date{October 2026}

\begin{document}
\maketitle

\begin{abstract}
We study the Euler--Maruyama approximation of the degenerate stochastic differential equation
\[
 dX_t=|X_t|^\alpha\,dW_t,\qquad 0<\alpha<\frac12,\qquad X_0=0,
\]
for which uniqueness in law fails.  The Euler scheme started exactly from zero is trapped by the absorbing solution, whereas a nonzero initialization $x_h\to0$ can select a different weak solution.  We identify the initialization-dependent limiting regimes.  The critical spatial scale is
\[
 r_h=h^{1/[2(1-\alpha)]},
\]
and the relevant depth parameter is
\[
 \Theta_h=h\log\left(1+\log^+\frac{r_h}{|x_h|}\right).
\]
If $\Theta_h\to0$, the scheme converges to the non-sticky weak solution.  If $\Theta_h\to\theta\in(0,\infty)$, it converges to the solution that stays at zero for the deterministic time $\theta/|\log\alpha|$ and then evolves as the non-sticky solution.  If $\Theta_h\to\infty$, the scheme converges to the absorbing solution.  The proof is based on an exact scale-invariant Euler chain, a logarithmic exit-time law of large numbers, a discrete Tanaka occupation estimate, and a compact-uniform restart argument.  Every finite deterministic delay can be obtained by a suitable vanishing initialization, while oscillating depth parameters yield process-level nonconvergence.  Thus a vanishing initialization can survive in the limit as a macroscopic deterministic delay and thereby act as a solution selector for the Euler--Maruyama approximation.
\end{abstract}



\section{Introduction}
Consider
\begin{equation}\label{eq:sde}
 dX_t=|X_t|^\alpha\,dW_t,\qquad 0<\alpha<\frac12,\qquad X_0=0.
\end{equation}
This equation is a classical one-dimensional example in which degeneracy of the diffusion coefficient destroys uniqueness in law.  Absolute-power coefficients of this form are not confined to abstract examples of non-uniqueness.  They have also appeared in free-boundary extensions of CEV/SABR-type models introduced to allow a modeled state variable to cross zero, notably in the negative-interest-rate literature; see, for example, \citet{AntonovKonikovSpector2015} and the application to Japanese interest-rate options in \citet{SuganumaYamada2017}.  Those models contain additional components, such as stochastic volatility, and are not covered by the results of the present paper.  The scalar equation \eqref{eq:sde} nevertheless isolates a basic issue that is independent of that particular application: once uniqueness in law fails at the degenerate point, specifying a continuous-time model and showing that a direct time discretization selects that model are distinct questions.  In particular, the existence of a prescribed free-boundary law does not by itself identify the vanishing-step limit of an Euler approximation of the underlying non-unique SDE.

More generally, \citet{EngelbertSchmidt1985} gave a description of weak solutions of driftless one-dimensional stochastic differential equations in terms of a fundamental solution, which has no sojourn time in the zero set of the diffusion coefficient, and time delays in that zero set.  Thus the existence of delayed weak solutions is part of the classical one-dimensional theory.  For the power-law coefficient in \eqref{eq:sde}, \citet{BassBurdzyChen2007} showed that weak uniqueness fails, whereas pathwise uniqueness holds within the class of solutions that spend zero Lebesgue time at the origin; in that class a strong solution exists.

Two basic weak solutions are useful to distinguish at the outset.  The identically zero process is an absorbing solution.  There is also a distinguished non-sticky solution $X^{\ns}$ satisfying
\[
 \int_0^T\1_{\{X_t^{\ns}=0\}}\,dt=0\quad\text{a.s.},\qquad T<\infty.
\]
Moreover, for every deterministic $d\ge0$, one may keep the process at zero up to time $d$ and then start an independent copy of the non-sticky solution.  These delayed solutions are weak solutions of the same SDE.  We emphasize that this deterministic initial holding is not the same object as a classical sticky diffusion, where positive occupation at a boundary is typically tied to local time or a speed-measure atom.  Sticky Brownian motion and related SDE representations have been studied, for example, by \citet{Bass2014} and \citet{EngelbertPeskir2014}; related time-change descriptions include \citet{Peskir2022} and \citet{RamirezUribe2024}.  Numerical approximation of prescribed sticky boundary dynamics is considered by \citet{BouRabeeHolmesCerfon2020} and, more recently, by \citet{Sharma2025}.  Our question is different: the limiting SDE itself does not specify a unique weak law, and we ask which law is selected by the ordinary Euler--Maruyama scheme.

Let $h>0$ and let $x_h\ne0$ with $x_h\to0$.  We consider
\begin{equation}\label{eq:em}
 X_t^h=x_h+\int_0^t|X_{\eta_h(s)}^h|^\alpha\,dW_s,
 \qquad \eta_h(s)=\lfloor s/h\rfloor h.
\end{equation}
If $x_h=0$, the scheme remains identically zero.  It is therefore natural to perturb the initialization.  The main point of this paper is that the statement ``$x_h\to0$'' is far from sufficient to determine the numerical limit.

The first relevant scale follows from balancing one Euler increment with the current position.  If $|x|$ is of order $r_h$, then
\[
 |x|^\alpha\sqrt h
\]
is of the same order as $|x|$ precisely when
\begin{equation}\label{eq:rh}
 r_h=h^{1/[2(1-\alpha)]},\qquad r_h^\alpha\sqrt h=r_h.
\end{equation}
At the grid points $Y_k^h=X_{kh}^h$, the rescaled process
\[
 V_k^h=Y_k^h/r_h
\]
satisfies the exact recursion
\begin{equation}\label{eq:rescaledchainintro}
 V_{k+1}^h=V_k^h+|V_k^h|^\alpha Z_{k+1},
\end{equation}
which no longer contains $h$.  The numerical selection problem is therefore reduced to the exit behavior of a single scale-invariant Markov chain started from $v_h=x_h/r_h$.
This exact reduction is special to the power-law coefficient and is deliberately exploited throughout the paper; we do not claim here a corresponding classification for general degenerate diffusion coefficients.

The chain has a logarithmic depth.  We use
\begin{equation}\label{eq:Hintro}
 H(v)=\log\left(1+\log^+\frac1{|v|}\right),\qquad v\ne0.
\end{equation}
For very small $v$, one step transforms $L=\log(1/|v|)$ approximately into $\alpha L$; consequently $H$ decreases by approximately $|\log\alpha|$ per step.  This suggests the physical exit time
\[
 h\tau_R(v_h)\approx \frac{hH(v_h)}{|\log\alpha|}.
\]
This leads to the selection parameter
\begin{equation}\label{eq:ThetaIntro}
 \Theta_h=hH(x_h/r_h)
 =h\log\left(1+\log^+\frac{r_h}{|x_h|}\right).
\end{equation}
Our main result is the following trichotomy.  If $\Theta_h\to0$, the microscopic trapping period disappears and the Euler approximation selects $X^{\ns}$.  If $\Theta_h\to\theta\in(0,\infty)$, the trapping period survives as the deterministic delay
\[
 d_\theta=\frac{\theta}{|\log\alpha|}.
\]
If $\Theta_h\to\infty$, the scheme remains close to zero on every compact time interval.  In particular, polynomial initializations are always subcritical.  A macroscopic delay requires a double-exponential initialization; for $d>0$, one may take
\[
 x_h=\exp\{-\exp(d|\log\alpha|/h)\}.
\]
Thus an initialization that vanishes faster than every algebraic scale can leave a finite temporal trace in the limit.

From a numerical point of view, the classification separates the mathematical selection mechanism from what is visible in a fixed floating-point format. Every initialization with
\[
  \log\log(1/|x_h|)=o(1/h)
\]
is subcritical and selects the non-sticky law. This class contains, in particular, all polynomial and algebraic initializations. By contrast, a finite macroscopic delay requires a double-exponential depth, which no fixed finite-precision format can represent for small $h$. For instance, in IEEE double precision the smallest positive (subnormal) number is approximately $4.94\times10^{-324}$. Hence, for $0<h\le1$ (so that $r_h\le1$), every representable $x_h\ne0$ satisfies
\[
  \log\frac{r_h}{|x_h|}\le\log\frac{1}{|x_h|}\le 744.5,
  \qquad\text{so}\qquad
  \Theta_h\le h\log(745.5)\le 6.62\,h .
\]
Thus the effective logarithmic depth $H(x_h/r_h)=\Theta_h/h$ remains uniformly bounded in fixed double precision. The theorem should therefore not be read as predicting several competing limits for an ordinary fixed-precision implementation. Rather, it identifies the exact mathematical boundary between initialization regimes. In particular, it explains why conventional nonzero perturbations select the non-sticky solution. The critical regimes can still be investigated without underflow by evolving the logarithmic depth variable in the exact representation of Appendix~\ref{app:exit}.

The occupation estimate is the second main ingredient.  In earlier work \citep{TaguchiTanaka2019}, we studied weak convergence of Euler--Maruyama schemes for degenerate SDEs under a non-sticky condition.  In the published version, Theorem 9.1 states weak convergence for the Euler scheme defined by (9.4), initialized at arbitrary $x_n\notin Z(\sigma)$ with $x_n\to x_0$, without a quantitative restriction on the rate of approach.  The proof of the key occupation estimate in Lemma 9.3, estimate (9.9), attempts to use the continuous occupation formula with the lagged Euler state $X_{\eta_h(s)}^h$.  For the present power-law model, the initialization-dependent limits show that this unrestricted initialization-independent conclusion cannot hold.  Here we replace that step by a discrete Tanaka argument applied directly to the Euler chain.  The resulting estimate is
\[
 \limsup_{h\downarrow0}\E\int_0^T\1_{\{|X_t^h|<\varepsilon\}}\,dt
 \le C_T\varepsilon^{1-2\alpha}
\]
in the subcritical regime.

The argument also yields a process-level cluster-set description for deterministic vanishing initializations.  In particular, every finite deterministic delay occurs as an Euler limit, and an oscillating selection parameter can make the numerical laws fail to converge even though $x_h\to0$.  Thus the contribution is not a construction of delayed weak solutions themselves, but a quantitative description of which such laws are selected by the Euler--Maruyama approximation as a function of the initialization rate.  To the best of our knowledge, this initialization-rate-dependent selection mechanism, and in particular the double-logarithmic selector above, has not previously been quantified for this scheme.

The result is related to, but distinct from, two neighboring approximation literatures.  Euler approximation with irregular or H\"older diffusion coefficients has been studied under conditions ensuring a prescribed limiting law; see, for example, \citet{Yan2002} and \citet{GyongyRasonyi2011}.  A different selection problem appears in the zero-noise regularization of non-unique Peano dynamics: \citet{BaficoBaldi1982} and \citet{DelarueFlandoli2014} identify laws or transition scales selected as an external noise amplitude vanishes.  Here the perturbation is neither a coefficient regularization nor a vanishing external noise: the SDE is fixed, and the selector is generated by the interaction between the Euler mesh and the rate at which the numerical initialization approaches the degenerate point.

The paper is organized as follows.  Section 2 states the main results.  Section 3 studies the scale-invariant chain and its exit time.  Section 4 controls the continuous interpolation before microscopic exit.  Section 5 develops the discrete occupation estimate.  Section 6 identifies the subcritical limit.  Section 7 proves the critical and supercritical regimes and the cluster-set result.  Appendices A--C contain the detailed exit-time, occupation, and stochastic-analysis arguments.

\section{Main results}
\subsection{Distinguished weak solutions}
Let $X^{\ns}$ denote the non-sticky weak solution of \eqref{eq:sde}.  For $d\in[0,\infty)$, let $(\widetilde X^{\ns},\widetilde W)$ be a copy of a non-sticky weak solution and let $B$ be an independent Brownian motion.  Define
\begin{equation}\label{eq:delay}
 X_t^{(d)}=\begin{cases}
 0,&0\le t\le d,\\
 \widetilde X_{t-d}^{\ns},&t>d,
 \end{cases}
 \qquad
 W_t^{(d)}=\begin{cases}
 B_t,&0\le t\le d,\\
 B_d+\widetilde W_{t-d},&t>d.
 \end{cases}
\end{equation}
Then $W^{(d)}$ is a Brownian motion in the joined filtration and
\[
 X_t^{(d)}=\int_0^t|X_s^{(d)}|^\alpha\,dW_s^{(d)}.
\]
Thus $X^{(d)}$ is a weak solution.  We have $X^{(0)}=X^{\ns}$ in law.  We also write $X^{(\infty)}\equiv0$.

\subsection{Euler scheme and selection parameter}
Let $x_h\ne0$, $x_h\to0$, and let $X^h$ solve \eqref{eq:em}.  Set
\[
 r_h=h^{1/[2(1-\alpha)]},\qquad v_h=x_h/r_h,
\]
\[
 H(v)=\log\left(1+\log^+\frac1{|v|}\right),\qquad v\ne0,\qquad c_\alpha=|\log\alpha|,
\]
and
\begin{equation}\label{eq:Theta}
 \Theta_h=hH(v_h).
\end{equation}

For the scale-one chain
\[
 V_{k+1}=V_k+|V_k|^\alpha Z_{k+1},
\]
define
\[
 \tau_R(v)=\inf\{k\ge0:|V_k|\ge R\},\qquad V_0=v.
\]

\begin{theorem}[Exit-time law of large numbers]\label{thm:exitLLN}
For every fixed $R>0$,
\[
 \frac{\tau_R(v)}{H(v)}\longrightarrow\frac1{c_\alpha}
\]
in probability as $v\to0$, $v\ne0$.  Moreover,
\[
 \E_v[\tau_R(v)]\le C_R\{1+H(v)\},\qquad 0<|v|<R.
\]
\end{theorem}

The next theorem is the discrete replacement for the occupation estimate used in our previous work \citep{TaguchiTanaka2019}.  Its subcritical assumption is substantive rather than technical: Theorem \ref{thm:selection} below shows that, without control of the initialization rate, delayed or absorbing limits may occur.

\begin{theorem}[Occupation estimate]\label{thm:occupation}
Let $h_n\downarrow0$ and let $x_n\ne0$ satisfy $x_n\to0$ and
\[
  \Theta_n:=h_nH(x_n/r_{h_n})\longrightarrow0 .
\]
Let $X^n$ denote the Euler approximation with mesh $h_n$ and initialization $x_n$. For every $T<\infty$ there is $C_T<\infty$, depending only on $T$, $\alpha$ and $\sup_n|x_n|$, such that, for $0<\varepsilon\le1$,
\[
  \limsup_{n\to\infty}
  \E\int_0^T\1_{\{|X^n_t|<\varepsilon\}}\,dt
  \le C_T\varepsilon^{1-2\alpha}.
\]
In particular,
\[
  \lim_{\varepsilon\downarrow0}\limsup_{n\to\infty}
  \E\int_0^T\1_{\{|X^n_t|<\varepsilon\}}\,dt=0 .
\]
Since the hypothesis is sequential, the same conclusion holds for the family $\{X^h\}$ whenever $\Theta_h\to0$ as $h\downarrow0$.
\end{theorem}

We use $C_{\loc}([0,\infty))$ with the local-uniform topology, for example with metric
\[
 d_{\loc}(f,g)=\sum_{m=1}^\infty2^{-m}(1\wedge\|f-g\|_{[0,m]}).
\]

\begin{theorem}[Selection theorem]\label{thm:selection}
Let $x_h\ne0$ and $x_h\to0$.
\begin{enumerate}[label=(\roman*)]
\item If $\Theta_h\to0$, then
\[
 X^h\Rightarrow X^{\ns}
\]
in $C_{\loc}([0,\infty))$.
\item If $\Theta_h\to\theta\in(0,\infty)$, then
\[
 X^h\Rightarrow X^{(d_\theta)},\qquad d_\theta=\theta/c_\alpha,
\]
in $C_{\loc}([0,\infty))$.
\item If $\Theta_h\to\infty$, then
\[
 X^h\longrightarrow0
\]
in probability in $C_{\loc}([0,\infty))$.
\end{enumerate}
\end{theorem}

\begin{corollary}[Every finite deterministic delay]\label{cor:everydelay}
For every $d\in[0,\infty)$ there exists a deterministic sequence $x_h\ne0$, $x_h\to0$, such that
\[
 X^h\Rightarrow X^{(d)}.
\]
For $d>0$, one may take
\[
 x_h=\exp\{-\exp(dc_\alpha/h)\}.
\]
For $d=0$, any polynomial initialization $x_h=h^\beta$, $\beta>0$, is sufficient.
\end{corollary}

\begin{corollary}[Cluster set]\label{cor:cluster}
Let $\mathcal C(\Theta)$ be the set of subsequential limits of $\Theta_h$ in $[0,\infty]$.  The cluster set of the laws $\{\mathcal L(X^h)\}$ in $\mathcal P(C_{\loc}([0,\infty)))$, endowed with weak convergence, is
\[
 \{\mathcal L(X^{(\theta/c_\alpha)}):\theta\in\mathcal C(\Theta)\cap[0,\infty)\}
 \cup
 \{\delta_0:\infty\in\mathcal C(\Theta)\}.
\]
Consequently the Euler laws converge if and only if $\mathcal C(\Theta)$ is a singleton.
\end{corollary}

\begin{remark}\label{rem:fixedhorizon}
On a fixed horizon $[0,T]$, all delays $d\ge T$ produce the zero path.  Distinct-delay and complete cluster-set statements are therefore naturally formulated on $C_{\loc}([0,\infty))$.
\end{remark}

\section{The rescaled Euler chain and its exit time}
At grid points,
\[
 Y_{k+1}^h=Y_k^h+|Y_k^h|^\alpha\sqrt h\,Z_{k+1}.
\]
With $V_k^h=Y_k^h/r_h$ and \eqref{eq:rh},
\begin{equation}\label{eq:chain}
 V_{k+1}^h=V_k^h+|V_k^h|^\alpha Z_{k+1}.
\end{equation}
Thus the transition mechanism is independent of $h$.

Let
\[
 \tau_R^h=\inf\{k\ge0:|V_k^h|\ge R\},\qquad S_h^R=h\tau_R^h.
\]
By Theorem \ref{thm:exitLLN}, whenever $v_h\to0$,
\begin{equation}\label{eq:physicalexit}
 \frac{S_h^R}{hH(v_h)}
 =\frac{\tau_R(v_h)}{H(v_h)}
 \longrightarrow\frac1{c_\alpha}
\end{equation}
in probability.  In the critical and supercritical regimes considered below,
$\Theta_h\to\theta>0$ or $\Theta_h\to\infty$ implies
$H(v_h)\to\infty$, and hence $v_h\to0$; thus the hypothesis of Theorem
\ref{thm:exitLLN} is automatic in precisely the regimes where this exit law
is used.  Since $hH(v_h)=\Theta_h$, the formula converts the number of
microscopic Euler steps into physical time.  If
$\Theta_h\to\theta\in(0,\infty)$, Slutsky's theorem gives
\[
 S_h^R
 =\Theta_h\frac{\tau_R(v_h)}{H(v_h)}
 \longrightarrow d_\theta=\theta/c_\alpha.
\]
If $\Theta_h\to\infty$, then for every fixed $T$ and every
$0<\delta<c_\alpha^{-1}$,
\begin{align*}
 \Pp(S_h^R\le T)
 &\le
 \Pp\left(
 \frac{\tau_R(v_h)}{H(v_h)}\le\delta
 \right)
 +\mathbf 1_{\{\delta\Theta_h\le T\}}
 \longrightarrow0.
\end{align*}
For $\Theta_h\to0$, formula \eqref{eq:physicalexit} alone need not be used
when $H(v_h)$ stays bounded; the subcritical proof instead relies on the
occupation estimate of Section 5, which covers all such initializations.

The depth $H$ also clarifies the initialization scales.  If $x_h=h^\beta$, then $\log(r_h/|x_h|)$ grows at most logarithmically in $1/h$, so
\[
 \Theta_h\to0.
\]
More generally, whenever
\[
 \log\log(1/|x_h|)=o(1/h),
\]
the initialization is subcritical.  A finite positive delay appears only when the double logarithm is of order $1/h$.

The proof of Theorem \ref{thm:exitLLN} is given in Appendix A.  Its key exact identity is obtained by writing $v=\xi e^{-L}$:
\[
 V_1=e^{-\alpha L}\left(Z+\xi e^{-(1-\alpha)L}\right).
\]
Thus the logarithmic depth changes by $\log\alpha+o_{L^2}(1)$ in the deep region.  A stopped martingale argument then yields the exit-time law of large numbers.

\section{Behavior before the microscopic exit}
For $kh\le t<(k+1)h$,
\[
 X_t^h=Y_k^h+|Y_k^h|^\alpha(W_t-W_{kh}).
\]
Before $\tau_R^h$, $|Y_k^h|<Rr_h$.  Brownian scaling and a union bound over $O(h^{-1})$ intervals show that the factor $r_h$ multiplying the maximal rescaled Brownian increment is sufficient to force the interpolation error to vanish.

\begin{proposition}[Pre-exit collapse]\label{prop:preexit}
For every $R,T>0$,
\[
 \sup_{0\le t\le S_h^R\wedge T}|X_t^h|\longrightarrow0
\]
in probability.
\end{proposition}
\begin{proof}
Put $N_h(T)=\lfloor T/h\rfloor$.  On $\{\tau_R^h=0\}$ we have $S_h^R=0$, so the left-hand side in the proposition equals $|x_h|$, which tends to zero.  It therefore remains to work on $\{\tau_R^h\ge1\}$.  If $k<\tau_R^h$, then
$|Y_k^h|<Rr_h$, and for $kh\le t<(k+1)h$,
\[
 |X_t^h|
 \le Rr_h+(Rr_h)^\alpha
       \sup_{kh\le s\le(k+1)h}|W_s-W_{kh}|.
\]
Since $r_h^\alpha\sqrt h=r_h$, Brownian scaling gives random variables
\[
 B_k^*=\sup_{0\le u\le1}|B_u^{(k)}|
\]
with the common law of $\sup_{0\le u\le1}|B_u|$ such that, on the part of the
$k$th interval preceding the exit,
\[
 \sup |X_t^h|\le r_h\{R+R^\alpha B_k^*\}.
\]
The terminal interval containing $S_h^R$ is covered by the same estimate up to
the exit time itself.  Hence
\begin{equation}\label{eq:preexit-max}
 \sup_{0\le t\le S_h^R\wedge T}|X_t^h|
 \le r_h\left\{R+R^\alpha
       \max_{0\le k\le N_h(T)}B_k^*\right\}.
\end{equation}
For every $a>0$, the reflection principle and a Gaussian tail bound imply
\[
 \Pp(B_k^*>a)\le 4\exp(-a^2/2).
\]
Therefore, for every $\varepsilon>0$ and all sufficiently small $h$,
\begin{align*}
 &\Pp\left(
 r_hR^\alpha\max_{0\le k\le N_h(T)}B_k^*>\varepsilon
 \right)\\
 &\qquad\le
 4\{N_h(T)+1\}
 \exp\left\{-\frac{\varepsilon^2}
 {2R^{2\alpha}r_h^2}\right\}.
\end{align*}
Because $N_h(T)+1=O(h^{-1})$ whereas
$r_h^{-2}=h^{-1/(1-\alpha)}$, the right-hand side tends to zero.  Since
$r_hR\to0$, \eqref{eq:preexit-max} proves the assertion.
\end{proof}

\begin{proposition}[Exit overshoot]\label{prop:overshoot}
For every $R>0$, the family
\[
 \{V_{\tau_R(v)}:0<|v|<R\}
\]
is tight.  Consequently, whenever $0<|x_h|<Rr_h$ eventually,
\[
 X_{S_h^R}^h=r_hV_{\tau_R^h}^h\longrightarrow0
\]
in probability.
\end{proposition}

The proofs are completed in Appendix A.  These propositions separate the microscopic trapping phase from the later macroscopic evolution: the process can spend a non-vanishing amount of physical time in a spatial region whose diameter tends to zero, yet it exits that region from a rescaled position bounded away from zero.

\section{Occupation near the degenerate point}
The quadratic variation of \eqref{eq:em} contains the lagged state $X_{\eta_h(s)}^h$, so we work directly with the Euler chain.

Set
\[
 s_h(x)=|x|^\alpha\sqrt h
\]
and
\[
 \mathcal R_h(x,a)=\E\bigl[|x+s_h(x)Z-a|-|x-a|\bigr].
\]
Here $\phi$ denotes the standard normal density and
\[
 \overline\Phi(q):=\int_q^\infty \phi(z)\,dz=1-\Phi(q)
\]
denotes the standard normal upper-tail function.  A Gaussian calculation gives
\[
 \mathcal R_h(x,a)=s_h(x)g\left(\frac{|x-a|}{s_h(x)}\right),
\]
where
\[
 g(q)=2\{\phi(q)-q\overline\Phi(q)\}>0.
\]
Conditional expectation and telescoping yield
\[
 \E\sum_{k=0}^{N-1}\mathcal R_h(Y_k^h,a)
 =\E|Y_N^h-a|-|x_h-a|.
\]

Choose once and for all $R_0>1$ so large that
\[
 2^\alpha R_0^{-(1-\alpha)}\le\frac14.
\]
For $\rho\ge R_0r_h$ and $\rho\le|x|<2\rho$, Appendix B proves the deterministic inequality
\begin{equation}\label{eq:localTanaka}
 h\le C\int_{\rho/2\le|a|\le3\rho}\mathcal R_h(x,a)|a|^{-2\alpha}\,da.
\end{equation}
Indeed, since $\sqrt h=r_h^{1-\alpha}$,
\[
 \frac{s_h(x)}{\rho}
 \le2^\alpha\left(\frac{r_h}{\rho}\right)^{1-\alpha}
 \le2^\alpha R_0^{-(1-\alpha)}\le\frac14.
\]
Thus only the local-window regime is needed: integrating over
$[x-s_h(x),x+s_h(x)]$ gives
\[
 s_h(x)^2\rho^{-2\alpha}
 =|x|^{2\alpha}h\rho^{-2\alpha}\ge c h.
\]
The lower cutoff $R_0r_h$ is essential; without it the displayed deterministic inequality need not hold uniformly in $h$ and $\rho$.

Summing \eqref{eq:localTanaka}, using the telescoping identity above, and integrating $|a|^{-2\alpha}$ over an interval of size $O(\rho)$ gives
\begin{equation}\label{eq:annulus}
 \E\left[h\sum_{k<N_h(T)}\1_{\{\rho\le|Y_k^h|<2\rho\}}\right]
 \le C_T\rho^{1-2\alpha}.
\end{equation}
This is the point at which $\alpha<1/2$ is essential.  Taking $R\ge2R_0$, a dyadic summation then controls the mesoscopic region $Rr_h\le|Y_k^h|<\varepsilon$; the smallest dyadic annulus still has lower radius at least $R_0r_h$.

The microscopic core requires a different argument because the expected number of steps in a single core excursion grows like $H(v)$.  Appendix B decomposes the path into successive entrances into and exits from $(-R,R)$ for the rescaled chain.  The first excursion is bounded by $hG_R(v_h)\le C_R(h+\Theta_h)$.  Later excursions are charged to the state immediately before entrance; returns from a bounded collar are controlled by the same discrete Tanaka remainder, while returns from the far field have a Gaussian cost.  This yields
\begin{equation}\label{eq:corebound}
 \E\left[h\sum_{k<N_h(T)}\1_{\{|Y_k^h|<Rr_h\}}\right]
 \le C_R(h+\Theta_h)+C_{T,R}\frac h{r_h}+C_{T,R}e^{-c_R/h}.
\end{equation}
Since
\[
 \frac h{r_h}=h^{(1-2\alpha)/(2(1-\alpha))}\to0,
\]
\eqref{eq:annulus} and \eqref{eq:corebound} imply the grid occupation estimate.  The continuous interpolation is recovered from
\[
 D_h(T)=\sup_{t\le T}|X_t^h-X_{\eta_h(t)}^h|\to0
\]
in probability.  This proves Theorem \ref{thm:occupation}; full details are in Appendix B.

\section{Identification of the non-sticky limit}
Uniform moment estimates and BDG give, for $p\ge2$, for some $h_0>0$,
\[
 \sup_{0<h\le h_0}\E\sup_{t\le T}|X_t^h|^p<\infty
\]
and
\[
 \E|X_t^h-X_s^h|^p\le C_{p,T}|t-s|^{p/2}.
\]
Since $p>2$ may be chosen, the increment estimate and Kolmogorov--Chentsov tightness criterion imply tightness in $C([0,T])$ for every fixed $T$.  Hence the laws are tight on every compact time interval.

For every $p>2$,
\begin{equation}\label{eq:mesh}
 \E[D_h(T)^p]\le C_{p,T}h^{p/2-1}.
\end{equation}
Because $2\alpha<1$,
\[
 \bigl||x|^{2\alpha}-|y|^{2\alpha}\bigr|\le|x-y|^{2\alpha},
\]
and \eqref{eq:mesh} implies
\[
 \E\int_0^T\left||X_{\eta_h(s)}^h|^{2\alpha}-|X_s^h|^{2\alpha}\right|ds\to0.
\]
Every subsequential limit therefore solves the martingale problem for
\[
 \mathcal Lf(x)=\frac12|x|^{2\alpha}f''(x),
\]
and hence is a weak solution of \eqref{eq:sde}.

Under $\Theta_h\to0$, Theorem \ref{thm:occupation} transfers to any weak limit and gives zero Lebesgue occupation at the origin.  The non-sticky law is unique.  For completeness, Appendix C gives the Brownian time-change proof: with
\[
 A_u=\int_0^u\1_{\{B_s\ne0\}}|B_s|^{-2\alpha}\,ds
\]
and inverse $T$, the process $B_{T_t}$ is the non-sticky weak solution, and every non-sticky weak solution has this law by the Dambis--Dubins--Schwarz theorem \citep{Dambis1965,DubinsSchwarz1965}.

\begin{lemma}[Sequential subcritical convergence]\label{lem:sequential-subcritical}
Let $h_n\downarrow0$ and let $x_n\ne0$ satisfy $x_n\to0$ and
\[
 h_nH(x_n/r_{h_n})\longrightarrow0.
\]
Then the Euler approximation with mesh $h_n$ and initialization $x_n$
converges to $X^{\ns}$ in $C_{\loc}([0,\infty))$.
\end{lemma}
\begin{proof}
Tightness and the martingale-problem identification
(Appendix~\ref{app:identification}) use only $h_n\downarrow0$ and the boundedness of
$(x_n)$. The occupation estimate is supplied by
Theorem~\ref{thm:occupation}, which is stated for sequences. Hence every
subsequential limit is a non-sticky weak solution, whose law is unique by
Appendix~\ref{app:identification}.
\end{proof}

\begin{theorem}\label{thm:subcritical}
If $\Theta_h\to0$, then
\[
 X^h\Rightarrow X^{\ns}
\]
in $C_{\loc}([0,\infty))$.
\end{theorem}
The preceding argument gives the convergence on every finite horizon; the standard local-uniform upgrade used here is stated and proved as Proposition \ref{prop:cloc} below.

We also need uniformity after microscopic exit.

\begin{proposition}[Compact-uniform restart]\label{prop:restart}
Fix $0<R<M<\infty$ and $T<\infty$.  Let $X^{h,v}$ be the Euler approximation started from $r_hv$.  For every bounded continuous $F:C([0,T])\to\R$,
\[
 \sup_{R\le|v|\le M}\left|\E F(X^{h,v})-\E F(X^{\ns})\right|\to0.
\]
\end{proposition}
\begin{proof}
Suppose the conclusion fails.  Then there are $\varepsilon>0$, a sequence
$h_n\downarrow0$, and $v_n$ with $R\le|v_n|\le M$ such that
\[
 \left|\E F(X^{h_n,v_n})-\E F(X^{\ns})\right|\ge\varepsilon
 \qquad(n\ge1).
\]
After passing to a subsequence, $v_n\to v_*$ with $R\le|v_*|\le M$.
For the triangular initialization $x_n=r_{h_n}v_n$ we have $x_n\to0$ and
\[
 h_n H(v_n)\longrightarrow0,
\]
because $H$ is bounded on $\{v:R\le|v|\le M\}$.  Lemma \ref{lem:sequential-subcritical} therefore applies to the triangular
initialization $x_n=r_{h_n}v_n$, and hence
\[
 X^{h_n,v_n}\Rightarrow X^{\ns}
 \qquad\text{in }C([0,T]).
\]
Since $F$ is bounded and continuous, the corresponding expectations converge,
contradicting the preceding lower bound.  Hence the convergence is uniform on
the compact set of rescaled initial states.
\end{proof}

\section{Critical and supercritical selection}
We now prove Theorem \ref{thm:selection}.  Fix $R>0$ and write $S_h^R=h\tau_R^h$.

\subsection{Subcritical regime}
This is Theorem \ref{thm:subcritical}.

\subsection{Critical regime}
Assume
\[
 \Theta_h\to\theta\in(0,\infty),\qquad d_\theta=\theta/c_\alpha.
\]
Theorem \ref{thm:exitLLN} gives
\begin{equation}\label{eq:Shcritical}
 S_h^R\to d_\theta
\end{equation}
in probability.  Proposition \ref{prop:preexit} shows that the path is uniformly negligible before $S_h^R$ on compact horizons, while Proposition \ref{prop:overshoot} shows that the exit position is $r_h$ times a tight rescaled variable.

We prove convergence separately according to the position of the observation
horizon relative to $d_\theta$.  If $T<d_\theta$, choose
$\eta>0$ such that $T+\eta<d_\theta$.  By \eqref{eq:Shcritical},
\[
 \Pp(S_h^R>T)\longrightarrow1.
\]
Consequently Proposition \ref{prop:preexit} gives
\[
 \sup_{0\le t\le T}|X_t^h|\longrightarrow0
\]
in probability.  This is exactly the restriction of $X^{(d_\theta)}$ to
$[0,T]$.

It remains to consider $T>d_\theta$.  Define the concatenation map
\[
 \Gamma_T(s,f)(t)=\begin{cases}
 f(0),&0\le t\le s,\\
 f(t-s),&s<t\le T,
 \end{cases}
 \qquad 0\le s\le T.
\]
In particular, if $f(0)=0$, this agrees with the delayed-path convention in
\eqref{eq:delay}.
After the stopping time $S_h^R$, the strong Markov property of Brownian
increments restarts the Euler scheme from
\[
 X_{S_h^R}^h=r_h\widehat V_h,
 \qquad \widehat V_h=V_{\tau_R^h}^h.
\]
By overshoot tightness, for every $\varepsilon>0$ one can choose $M$ so that
\[
 \limsup_{h\downarrow0}\Pp(|\widehat V_h|>M)<\varepsilon.
\]
Moreover $|\widehat V_h|\ge R$ by definition of the exit time.  Let
$\widehat X^h$ denote the restarted Euler path, shifted so that its new time
origin is zero.  On $R\le|\widehat V_h|\le M$, Proposition
\ref{prop:restart} applies uniformly.

To spell out the resulting joint convergence, let $\Psi$ be bounded and
Lipschitz on $[0,T]\times C([0,T])$, and define
\[
 g(s):=\E\Psi(s,X^{\ns}),\qquad 0\le s\le T.
\]
Since $\Psi$ is Lipschitz, $|g(s)-g(s')|\le\operatorname{Lip}(\Psi)|s-s'|$; in particular $g$ is bounded and continuous. Since $d_\theta<T$,
\eqref{eq:Shcritical} implies $\Pp(S_h^R>T)\to0$. Conditioning at $S_h^R$
and using the independence of the future Brownian increments, together with
the compact-uniform restart estimate on $R\le|\widehat V_h|\le M$, gives
\begin{align*}
 &\left|\E\Psi(S_h^R\wedge T,\widehat X^h)
       -\E g(S_h^R\wedge T)\right|\\
 &\qquad\le o(1)
   +2\|\Psi\|_\infty\Pp(|\widehat V_h|>M)
   +2\|\Psi\|_\infty\Pp(S_h^R>T),
\end{align*}
where the $o(1)$ is uniform for $R\le|\widehat V_h|\le M$.  To justify also the uniformity in the first argument of $\Psi$, choose a finite mesh $s_1,\ldots,s_J$ of $[0,T]$ and let
$\delta_J=\sup_{s\in[0,T]}\min_j|s-s_j|$.  Then
\begin{align*}
 &\sup_{s\in[0,T]}\sup_{R\le|v|\le M}
 \left|\E\Psi(s,X^{h,v})-\E\Psi(s,X^{\ns})\right|\\
 &\qquad\le 2\operatorname{Lip}(\Psi)\delta_J
 +\max_{1\le j\le J}\sup_{R\le|v|\le M}
 \left|\E\Psi(s_j,X^{h,v})-\E\Psi(s_j,X^{\ns})\right|.
\end{align*}
The first term can be made arbitrarily small by refining the mesh, while the second tends to zero by Proposition \ref{prop:restart}.  First letting $h\downarrow0$ and then
$M\uparrow\infty$, and using $S_h^R\to d_\theta$ in probability, gives
\begin{equation}\label{eq:jointrestart}
 (S_h^R\wedge T,\widehat X^h)\Rightarrow(d_\theta,X^{\ns}).
\end{equation}
Here the second component on the right may be taken independent of the
constant first component; no independence assertion about the two prelimit
components is required.

Notice first that $\Gamma_T(s,f)\in C([0,T])$ for every
$f\in C([0,T])$: at the concatenation time $t=s$, both one-sided values
equal $f(0)$.  In fact, writing
$\Gamma_T(s,f)(t)=f((t-s)^+)$ and denoting by $\omega_f$ the modulus of
continuity of $f$, we have
\[
 \|\Gamma_T(s_n,f_n)-\Gamma_T(s,f)\|_{[0,T]}
 \le \|f_n-f\|_{[0,T]}+\omega_f(|s_n-s|).
\]
Thus $\Gamma_T$ is continuous on $[0,T]\times C([0,T])$ (with the
uniform topology on the second factor), in particular at
$(d_\theta,X^{\ns})$.
The actual path before $S_h^R$ is not constant at its exit value, but
Proposition \ref{prop:preexit} shows that this discrepancy is negligible:
with the natural concatenated path
$\widetilde X^h=\Gamma_T(S_h^R\wedge T,\widehat X^h)$,
\[
 \|X^h-\widetilde X^h\|_{[0,T]}
 \le
 2\sup_{0\le t\le S_h^R\wedge T}|X_t^h|
 +|X_{S_h^R}^h|\,\1_{\{S_h^R>T\}}
 \longrightarrow0
\]
in probability, by Proposition \ref{prop:overshoot} for the extra term.
On the other hand, \eqref{eq:jointrestart} and the
continuous mapping theorem give
\[
 \widetilde X^h
 =\Gamma_T(S_h^R\wedge T,\widehat X^h)
 \Rightarrow
 \Gamma_T(d_\theta,X^{\rm ns})
 =X^{(d_\theta)}
 \qquad\text{in }C([0,T]).
\]
Since $X^h$ and $\widetilde X^h$ are defined on the same prelimit
probability space and their uniform distance converges to zero in
probability, Slutsky's theorem therefore yields
\[
 X^h\Rightarrow X^{(d_\theta)}
 \qquad\text{in }C([0,T]),\qquad T>d_\theta.
\]
Finally, when $T=d_\theta$, choose any $T'>d_\theta$, apply the preceding
argument on $[0,T']$, and then use continuity of the restriction map
$C([0,T'])\to C([0,T])$.  Thus the critical convergence holds on every
finite horizon.

\subsection{Supercritical regime}
Assume $\Theta_h\to\infty$.  For every fixed $T$ and $R$,
\[
 \Pp(S_h^R\le T)\to0.
\]
On the complementary event, Proposition \ref{prop:preexit} gives
\[
 \sup_{t\le T}|X_t^h|\to0
\]
in probability.  Hence
\[
 X^h\to0
\]
in probability in $C([0,T])$ for every $T$.

\subsection{Local-uniform topology}
\begin{proposition}\label{prop:cloc}
Suppose that for every integer $m\ge1$, the restrictions $X^h|_{[0,m]}$ converge weakly to $X|_{[0,m]}$ in $C([0,m])$.  Then
\[
 X^h\Rightarrow X
\]
in $C_{\loc}([0,\infty))$.
\end{proposition}
\begin{proof}
Take an arbitrary sequence $h_n\downarrow0$.  For each $m$, tightness of the restrictions gives a compact set $K_m\subset C([0,m])$ with complement probability at most $\varepsilon2^{-m}$, after enlarging $K_m$ to absorb finitely many exceptional initial terms.  The set
\[
 K=\bigcap_{m\ge1}\pi_m^{-1}(K_m)
\]
is compact in the local-uniform topology by a diagonal Arzel\`a--Ascoli argument, and a union bound gives tightness.  Any subsequential limit has all finite-horizon restrictions equal in law to those of $X$.  Since the Borel sigma-field is generated by the restriction maps, the limit law is uniquely determined.
\end{proof}

The subcritical convergence of Theorem \ref{thm:subcritical} and the finite-horizon critical convergence proved above therefore yield the corresponding assertions in $C_{\loc}([0,\infty))$ by Proposition \ref{prop:cloc}.  In the supercritical regime we retain the stronger mode of convergence stated in Theorem \ref{thm:selection}.  Indeed, for every integer $M\ge1$,
\[
 \|X^h\|_{[0,M]}\longrightarrow0
 \qquad\text{in probability}.
\]
Given $\varepsilon>0$, choose $M$ so large that $\sum_{m>M}2^{-m}<\varepsilon/2$.  Then, by the definition of $d_{\loc}$,
\[
 \Pp\bigl(d_{\loc}(X^h,0)>\varepsilon\bigr)
 \le
 \Pp\left(\sum_{m=1}^{M}2^{-m}(1\wedge\|X^h\|_{[0,m]})>\varepsilon/2\right),
\]
and the probability on the right tends to zero because the finite sum converges to zero in probability.  Hence $X^h\to0$ in probability in $C_{\loc}([0,\infty))$.

\subsection{Prescribed delays and cluster sets}
For $d>0$, take
\[
 x_h=\exp\{-\exp(dc_\alpha/h)\}.
\]
The polynomial factor $r_h$ is negligible after taking two logarithms, and hence
\[
 \Theta_h\to dc_\alpha.
\]
For $d=0$, take for example $x_h=h^\beta$, $\beta>0$; then $\Theta_h\to0$.  This proves Corollary \ref{cor:everydelay}.

For an arbitrary initialization sequence, every subsequence of $\Theta_h$ has a further subsequence converging in the compactified interval $[0,\infty]$.  The three parts of Theorem \ref{thm:selection} identify the corresponding process limit.  Conversely every cluster point of $\Theta_h$ is realized by a subsequence.

It remains only to note that different delay parameters give different laws on $C_{\loc}([0,\infty))$.  Let $0\le d_1<d_2<\infty$ and choose $t\in(d_1,d_2)$.  With
\[
 F_t(f)=\sup_{0\le s\le t}|f(s)|,
\]
we have $F_t(X^{(d_2)})=0$ almost surely.  On the other hand,
\[
 \Pp\bigl(F_t(X^{(d_1)})>0\bigr)>0,
\]
because after time $d_1$ the process has the non-sticky law, whose zero set has Lebesgue measure zero almost surely and hence cannot contain an interval of positive length.  Hence $\mathcal L(X^{(d_1)})\ne\mathcal L(X^{(d_2)})$.  The absorbing law $\delta_0$ is likewise distinct from every finite-delay law.  This proves Corollary \ref{cor:cluster} and the stated convergence criterion.  In particular, two different subsequential limits of $\Theta_h$ produce different process laws and hence nonconvergence of the full Euler sequence.

\subsection{Relation to the earlier non-sticky convergence statement}
The present example also isolates the role of initialization in the convergence statement of \citet{TaguchiTanaka2019}.  Three points are worth separating.

First, the Euler scheme in \citet{TaguchiTanaka2019}, equation (9.4) of the published version, is initialized at $x_n\notin Z(\sigma)$ with $x_n\to x_0$, without a quantitative rate restriction. Here the limiting initial value $x_0$ is allowed to belong to $Z(\sigma)$. Theorem 9.1 then asserts weak convergence to the non-sticky law.  In the present power-law model this is not sufficient for such convergence.  Indeed, for any $0<d<T$, choosing an initialization with $\Theta_h\to dc_\alpha$ gives, by Theorem \ref{thm:selection}, the delayed limit $X^{(d)}$ rather than $X^{\ns}$ on $[0,T]$.  Thus the present model supplies an explicit counterexample to the unrestricted initialization-independent conclusion of Theorem 9.1 under merely $x_h\ne0$ and $x_h\to0$; the power-law coefficient itself is the setting of Corollary 9.1 in that paper.

Second, the precise obstruction in the earlier occupation step is that the continuous Euler interpolation has quadratic variation density $|X_{\eta_h(s)}^h|^{2\alpha}$, whereas the desired singular occupation estimate is expressed in the current state $X_s^h$.  The continuous occupation-density formula therefore does not supply the required weight $|X_s^h|^{-2\alpha}$ for the lagged-state quantity without an additional comparison that is unavailable near the degeneracy.  The non-sticky conclusion is nevertheless recovered in the subcritical regime $\Theta_h\to0$: Theorem \ref{thm:occupation}, proved through the discrete Tanaka estimate and Lemma \ref{lem:core-occupation}, replaces that step.  Combined with the martingale-problem and time-change arguments of Section~6 and Appendix~C, it yields $X^h\Rightarrow X^{\ns}$.

Third, the subcritical condition is sharp at the level of the selection parameter.  Finite positive limits of $\Theta_h$ produce finite deterministic delays, while $\Theta_h\to\infty$ produces the absorbing limit.  Hence the rate condition is not introduced merely to repair an estimate: it identifies the boundary between genuinely different numerical limits.  Thus, in the present power-law model, the result identifies the rate condition missing from the unrestricted initialization-independent conclusion and explains the phenomenon that the earlier occupation argument could not detect.

\section{Conclusion}
We have shown that the Euler--Maruyama approximation of
\[
 dX_t=|X_t|^\alpha dW_t,\qquad 0<\alpha<\frac12,
\]
exhibits an initialization-dependent solution-selection phase transition.  For deterministic vanishing initialization sequences converging to the degenerate point zero, this phase transition is completely characterized by the depth parameter $\Theta_h$.  The critical spatial scale is $r_h=h^{1/[2(1-\alpha)]}$, while the effective selector is the logarithmic depth
\[
 \Theta_h=h\log\left(1+\log^+\frac{r_h}{|x_h|}\right).
\]
Subcritical initialization selects the non-sticky solution, critical initialization leaves a finite deterministic holding time, and supercritical initialization selects the absorbing solution.  Every finite deterministic delay can be generated, and oscillating depths can destroy convergence at process level.

The mechanism separates two effects that are invisible if one only records $x_h\to0$: the scale-invariant microscopic Euler chain determines how many numerical steps are required to escape the degeneracy, while the factor $h$ converts that number of steps into physical time.  The discrete Tanaka estimate then controls repeated returns and identifies the non-sticky dynamics after escape.

The exact scale invariance used in the analysis is specific to the power-law coefficient.  At the same time, absolute-power coefficients have arisen in zero-crossing diffusion models, including free-boundary CEV/SABR-type specifications.  The selection phenomenon studied here therefore concerns the numerical treatment of a degenerate boundary rather than a feature tied to a particular application.  It also suggests a broader question: for coefficients whose local behavior near a degenerate point is comparable, in an explicitly quantified sense, to $c|x|^\alpha$, does an analogous depth variable govern numerical solution selection?  Establishing such a result would require stability of both the microscopic exit law and the discrete occupation estimates under perturbations of the coefficient.

\appendix
\section{Exit-time estimates for the rescaled chain}\label{app:exit}
Consider
\[
 V_{k+1}=V_k+|V_k|^\alpha Z_{k+1},\qquad 0<\alpha<1.
\]
For $R>0$, let $\tau_R(v)=\inf\{k\ge0:|V_k|\ge R\}$.

\subsection{One-step logarithmic representation}
Write $v=\xi e^{-L}$, $\xi\in\{-1,1\}$.  Then
\[
 V_1=e^{-\alpha L}\left(Z_1+\xi e^{-(1-\alpha)L}\right).
\]
Set
\[
 U_L^{(\xi)}=-\log|Z_1+\xi e^{-(1-\alpha)L}|.
\]
A shifted Gaussian density gives, uniformly for sufficiently large $L$,
\[
 \Pp(U_L^{(\xi)}>u)\le Ce^{-u},\qquad
 \Pp(U_L^{(\xi)}<-u)\le Ce^{-ce^{2u}}.
\]
Hence all polynomial moments are uniformly bounded and
\[
 U_L^{(\xi)}\to-\log|Z_1|
\]
in $L^q$ for every finite $q$.

The chain started from a nonzero point hits zero at a grid time with probability zero, because every one-step transition from a nonzero point has a density.

\subsection{Depth drift}
For $H(v)=\log(1+\log^+(1/|v|))$, $v\ne0$,
\[
 H(V_1)-H(v)\to\log\alpha
\]
in $L^2$, uniformly in the sign of $v$, as $v\to0$.  To justify the $L^2$ statement, write $L=\log(1/|v|)$ and $U_L=U_L^{(\xi)}$.  On $\{\alpha L+U_L>0\}$,
\[
 H(V_1)-H(v)=\log\frac{1+\alpha L+U_L}{1+L}.
\]
On $\{|U_L|\le\alpha L/2\}$ a Taylor bound, together with the uniform moment bounds above, gives convergence in $L^2$ to $\log\alpha$.  The complementary event is negligible by the two tail estimates: in particular $\Pp(\alpha L+U_L\le0)\le C\exp\{-ce^{2\alpha L}\}$, which absorbs the possible factor $\log^2(1+L)$, while the positive tail is exponentially small.  The same estimates give uniform integrability of the squared logarithmic increment.  Therefore, for every $\delta>0$, there exists $r_\delta>0$ such that for $0<|v|\le r_\delta$,
\[
 \left|\E_v[H(V_1)-H(v)]+c_\alpha\right|\le\delta,
\]
and the second moment of the increment is uniformly bounded.

\subsection{Stopped martingale and exit LLN}
Fix small $r$ and let $\sigma_r=\inf\{k:|V_k|\ge r\}$.  Stopping first at $n\wedge\sigma_r$ avoids any circular integrability assumption.  If $D_k=H(V_{(k+1)\wedge\sigma_r})-H(V_{k\wedge\sigma_r})$, then
\[
 H(V_{n\wedge\sigma_r})=H(v)-c_\alpha(n\wedge\sigma_r)+\sum_{k<(n\wedge\sigma_r)}\varepsilon_k+M_{n\wedge\sigma_r},
\]
where $|\varepsilon_k|\le\delta$ and
\[
 \E[M_{n\wedge\sigma_r}^2]\le C\E[n\wedge\sigma_r].
\]
Taking expectations and using $H\ge0$ gives
\[
 \E[n\wedge\sigma_r]\le C_r(1+H(v)).
\]
Monotone convergence yields
\[
 \E_v\sigma_r\le C_r(1+H(v)).
\]
We now spell out the early- and late-exit estimates.  Put $H_0=H(v)$ and fix $\varepsilon>0$.  Choose $\delta>0$ so small that
\[
 \kappa_-:=1-(c_\alpha+\delta)\left(\frac1{c_\alpha}-\varepsilon\right)>0,
 \qquad
 \kappa_+:=(c_\alpha-\delta)\left(\frac1{c_\alpha}+\varepsilon\right)-1>0.
\]
Let
\[
 n_-:=\left\lfloor\left(\frac1{c_\alpha}-\varepsilon\right)H_0\right\rfloor,
 \qquad
 n_+:=\left\lceil\left(\frac1{c_\alpha}+\varepsilon\right)H_0\right\rceil.
\]
Since $|V_{\sigma_r}|\ge r$, we have $H(V_{\sigma_r})\le H(r)$.  On $\{\sigma_r\le n_-\}$, the stopped decomposition therefore implies
\[
 \begin{aligned}
 M_{\sigma_r}
 &=H(V_{\sigma_r})-H_0+c_\alpha\sigma_r-
   \sum_{k<\sigma_r}\varepsilon_k\\
 &\le H(r)-H_0+(c_\alpha+\delta)n_-\\
 &\le H(r)-\kappa_-H_0+O(1).
 \end{aligned}
\]
Thus, for all sufficiently small $|v|$,
\[
 \{\sigma_r\le n_-\}
 \subset
 \left\{|M_{n_-\wedge\sigma_r}|\ge\frac{\kappa_-}{2}H_0\right\}.
\]
Using the martingale second-moment estimate and the bound
$\E[n_-\wedge\sigma_r]\le C_r(1+H_0)$, Chebyshev's inequality yields
\[
 \Pp_v(\sigma_r\le n_-)
 \le
 \frac{C_r(1+H_0)}{H_0^2}
 \longrightarrow0.
\]

For the late exit, on $\{\sigma_r>n_+\}$ we have $n_+\wedge\sigma_r=n_+$ and $H(V_{n_+})\ge0$.  Hence
\[
 \begin{aligned}
 M_{n_+}
 &=H(V_{n_+})-H_0+c_\alpha n_+-
   \sum_{k<n_+}\varepsilon_k\\
 &\ge -H_0+(c_\alpha-\delta)n_+\\
 &\ge \kappa_+H_0+O(1).
 \end{aligned}
\]
Consequently, for all sufficiently small $|v|$,
\[
 \{\sigma_r>n_+\}
 \subset
 \left\{|M_{n_+\wedge\sigma_r}|\ge\frac{\kappa_+}{2}H_0\right\},
\]
and another application of Chebyshev's inequality gives
\[
 \Pp_v(\sigma_r>n_+)
 \le
 \frac{C_r(1+H_0)}{H_0^2}
 \longrightarrow0.
\]
Combining the two estimates proves
\[
 \frac{\sigma_r(v)}{H(v)}\to\frac1{c_\alpha}
\]
in probability.

It remains to pass from the deep exit $\sigma_r$ to $\tau_R$.  Put
$C_{r,R}=\{x:r\le|x|<R\}$.  Since the one-step Gaussian kernel is nondegenerate and the exit probability extends continuously to the compact closure $\{x:r\le|x|\le R\}$,
\[
 p:=\inf_{x\in C_{r,R}}\Pp_x(|V_1|\ge R)>0.
\]
A failed exit attempt may, however, jump back into $(-r,r)$, so the preceding lower bound alone is not a geometric argument.  We control such returns using the already proved estimate for $\sigma_r$.
For $N\ge1$ let $F_N(x)=\E_x(\tau_R\wedge N)$ and
$M_N=\sup_{x\in C_{r,R}}F_N(x)$.  By the first-step decomposition, the strong Markov property at the next return to $C_{r,R}$, and
$\E_y\sigma_r\le C_r(1+H(y))$ for $0<|y|<r$,
\[
 M_N\le1+(1-p)M_N+K_{r,R},
\]
where
\[
 K_{r,R}:=C_r\sup_{x\in C_{r,R}}
 \E_x\!\left[(1+H(V_1))\1_{\{0<|V_1|<r\}}\right]<\infty.
\]
Indeed, for $x\in C_{r,R}$ the one-step transition density satisfies
\[
 p_x(y)\le \frac{1}{\sqrt{2\pi}\,r^\alpha},\qquad y\in\R.
\]
Moreover, since $r<1$, the substitution $s=\log(1/y)$ gives
\[
 \int_0^r H(y)\,dy
 =\int_{\log(1/r)}^\infty \log(1+s)e^{-s}\,ds<\infty.
\]
Therefore
\[
 K_{r,R}\le \frac{C_r}{\sqrt{2\pi}\,r^\alpha}
 \int_{-r}^r(1+H(y))\,dy<\infty.
\]
Hence
$M_N\le(1+K_{r,R})/p$ uniformly in $N$, and monotone convergence gives
\[
 \sup_{x\in C_{r,R}}\E_x\tau_R<\infty.
\]
Applying the strong Markov property at $\sigma_r$ now yields
\[
 \E_v(\tau_R-\sigma_r)\le C_{r,R}.
\]
Consequently $(\tau_R-\sigma_r)/H(v)\to0$ in probability as $v\to0$, and together with the preceding estimates,
\[
 \frac{\tau_R(v)}{H(v)}\to\frac1{c_\alpha},\qquad
 \E_v\tau_R(v)\le C_R(1+H(v)).
\]

\subsection{Overshoot tightness}
The overshoot admits a direct uniform estimate that does not require a uniform bound on the total Green mass.  Let
\[
 \widehat V=V_{\tau_R},\qquad 0<|v|<R.
\]
On the event $\{\tau_R=k+1\}$, the pre-exit state satisfies $|V_k|<R$.  Hence, for $M>R$,
\[
 \{|\widehat V|>M,\ \tau_R=k+1\}
 \subset
 \left\{|Z_{k+1}|>\frac{M-R}{|V_k|^\alpha},\ \tau_R>k\right\}
 \subset
 \left\{|Z_{k+1}|>\frac{M-R}{R^\alpha},\ \tau_R>k\right\}.
\]
This bound alone cannot be summed over all $k$, because $\E_v\tau_R$ diverges as $v\to0$.  We therefore compare a large overshoot with an ordinary exit at the same step.  For $|u|<R$ put
\[
 p_M(u)=\Pp(|u+|u|^\alpha Z|>M),\qquad
 p_R(u)=\Pp(|u+|u|^\alpha Z|\ge R).
\]
For $M>2R$, Gaussian tail estimates imply
\begin{equation}\label{eq:overshoot-ratio}
 \sup_{|u|<R}\frac{p_M(u)}{p_R(u)}\longrightarrow0,
 \qquad M\to\infty.
\end{equation}
To verify this uniformly, fix once and for all $r\in(0,R/2)$.  On the compact collar $r\le|u|<R$, the denominator $p_R(u)$ is bounded away from zero, whereas $p_M(u)\to0$ uniformly as $M\to\infty$ by Gaussian tails.  Hence the ratio tends to zero uniformly on the collar.  On the inner region $0<|u|<r$, set
\[
  a_M(u)=\frac{M-|u|}{|u|^\alpha},\qquad
  b_R(u)=\frac{R+|u|}{|u|^\alpha}.
\]
Write $\overline\Phi$ for the standard normal upper-tail function. The event that the Gaussian increment crosses the boundary on the side opposite to $u$ gives $p_R(u)\ge\overline\Phi(b_R(u))$, whereas $p_M(u)\le2\overline\Phi(a_M(u))$.

Since $M>2R$ and $|u|<r<R/2$, we have $M-|u|>R+|u|$, that is, $a_M(u)>b_R(u)$. Moreover $b_R(u)\ge R/r^\alpha$. The Mills inequalities
\[
  \frac{b\,\phi(b)}{1+b^2}\le\overline\Phi(b),\qquad
  \overline\Phi(a)\le\frac{\phi(a)}{a}
\]
therefore give
\[
  \frac{p_M(u)}{p_R(u)}
  \le 2\,\frac{1+b_R^2}{a_M b_R}\,
      e^{-(a_M^2-b_R^2)/2}
  \le 2\Bigl(1+\frac{r^{2\alpha}}{R^2}\Bigr)
      \exp\Bigl\{-\frac{(M-|u|)^2-(R+|u|)^2}{2|u|^{2\alpha}}\Bigr\}.
\]
For $0<|u|<r$, the exponent is bounded above by
\[
  -\frac{(M-r)^2-(R+r)^2}{2r^{2\alpha}},
\]
which tends to $-\infty$ quadratically in $M$. Thus the ratio tends to zero uniformly also on $0<|u|<r$. Combining the two regions proves \eqref{eq:overshoot-ratio}.

Now let $\mathsf G_R(v,du)$ denote the killed Green measure of the chain before $\tau_R$.  Since exit occurs exactly once,
\[
 1=\Pp_v(\tau_R<\infty)
   =\int_{(-R,R)}\mathsf G_R(v,du)\,p_R(u),
\]
where finiteness of $\tau_R$ follows from the preceding exit-time estimate.  Consequently,
\[
 \begin{aligned}
 \Pp_v(|V_{\tau_R}|>M)
 &=\int_{(-R,R)}\mathsf G_R(v,du)\,p_M(u)\\
 &\le
 \sup_{|u|<R}\frac{p_M(u)}{p_R(u)}
 \int_{(-R,R)}\mathsf G_R(v,du)\,p_R(u)\\
 &=\sup_{|u|<R}\frac{p_M(u)}{p_R(u)}.
 \end{aligned}
\]
By \eqref{eq:overshoot-ratio},
\[
 \lim_{M\to\infty}\sup_{0<|v|<R}
 \Pp_v(|V_{\tau_R}|>M)=0.
\]
This proves Proposition \ref{prop:overshoot} without requiring a uniform bound on the total Green mass.

\section{Discrete occupation estimates}\label{app:occupation}
We prove Theorem \ref{thm:occupation} entirely at the level of the Euler chain.  This point is important because the quadratic variation of the continuous Euler interpolation contains the lagged variable $X_{\eta_h(s)}^h$; a continuous occupation-density formula for $X_s^h$ therefore does not directly estimate the quantity needed here.

\subsection{The discrete Tanaka remainder}
Put
\[
 s_h(x)=|x|^\alpha\sqrt h,
 \qquad
 \mathcal R_h(x,a)=\E[|x+s_h(x)Z-a|-|x-a|].
\]
If $s_h(x)>0$, writing $q=|x-a|/s_h(x)$ and integrating the two Gaussian tails gives
\[
 \mathcal R_h(x,a)=s_h(x)g(q),
 \qquad
 g(q)=2\{\phi(q)-q\overline\Phi(q)\}
      =2\int_q^\infty(z-q)\phi(z)\,dz.
\]
In particular $g$ is continuous and strictly positive on $[0,\infty)$.  Conditional on $\mathcal F_{kh}$,
\[
 \E[|Y_{k+1}^h-a|-|Y_k^h-a|\mid\mathcal F_{kh}]
 =\mathcal R_h(Y_k^h,a).
\]
Consequently, for every integer $N\ge1$,
\begin{equation}\label{eq:B-telescope}
 \E\sum_{k=0}^{N-1}\mathcal R_h(Y_k^h,a)
 =\E|Y_N^h-a|-|x_h-a|.
\end{equation}
The moment estimate of Appendix \ref{app:identification} implies, uniformly for $N\le N_h(T)+1$ and $|a|\le3$,
\begin{equation}\label{eq:B-telescope-bound}
 \E\sum_{k=0}^{N-1}\mathcal R_h(Y_k^h,a)\le C_T.
\end{equation}

\subsection{A deterministic annular inequality}
We now prove \eqref{eq:localTanaka} without suppressing the scale comparison on which it depends.  Suppose
\[
 0<\rho\le1,
 \qquad
 \rho\le|x|<2\rho,
 \qquad
 s=s_h(x)=|x|^\alpha\sqrt h.
\]
Choose $R_0>1$ so large that
\[
 2^\alpha R_0^{-(1-\alpha)}\le\frac14.
\]
We claim that for every $\rho\ge R_0r_h$ and every
$\rho\le|x|<2\rho$,
\begin{equation}\label{eq:B-local-ineq}
 h\le C\int_{\rho/2\le|a|\le3\rho}
 \mathcal R_h(x,a)|a|^{-2\alpha}\,da.
\end{equation}
Indeed, using $\sqrt h=r_h^{1-\alpha}$,
\[
 \frac{s}{\rho}
 =\frac{|x|^\alpha r_h^{1-\alpha}}{\rho}
 \le2^\alpha\left(\frac{r_h}{\rho}\right)^{1-\alpha}
 \le2^\alpha R_0^{-(1-\alpha)}\le\frac14.
\]
Set $I_x=[x-s,x+s]$.  Since $s\le\rho/4$ and
$\rho\le|x|<2\rho$, the interval $I_x$ has the same sign as $x$ and
\[
 \frac{3\rho}{4}\le |a|\le\frac{9\rho}{4}
 \qquad(a\in I_x),
\]
so $I_x\subset\{\rho/2\le|a|\le3\rho\}$.  Moreover
$|x-a|/s\le1$ on $I_x$.  With
\[
 c_0=\min_{0\le q\le1}g(q)>0,
\]
we have $\mathcal R_h(x,a)\ge c_0s$ on $I_x$, and therefore
\begin{align}
 \int_{\rho/2\le|a|\le3\rho}
 \mathcal R_h(x,a)|a|^{-2\alpha}\,da
 &\ge c_0s(2s)(3\rho)^{-2\alpha}\notag\\
 &=c_1|x|^{2\alpha}h\rho^{-2\alpha}\notag\\
 &\ge c_1h.                                      \label{eq:B-case1}
\end{align}
This proves \eqref{eq:B-local-ineq}.  Notice that the restriction
$\rho\ge R_0r_h$ is not cosmetic: when $s_h(x)$ is much larger than
$\rho$, the integral in \eqref{eq:B-local-ineq} is of order
$s_h(x)\rho^{1-2\alpha}$ and need not dominate $h$ uniformly.

\subsection{Annular and mesoscopic occupation}
For $R_0r_h\le\rho\le1$, apply \eqref{eq:B-local-ineq} to $x=Y_k^h$, sum over $k<N_h(T)$, and use Tonelli together with \eqref{eq:B-telescope-bound}.  Since $|a|\le3\rho\le3$,
\begin{align}
 \E\left[h\sum_{k<N_h(T)}\1_{\{\rho\le|Y_k^h|<2\rho\}}\right]
 &\le C_T\int_{\rho/2\le|a|\le3\rho}|a|^{-2\alpha}\,da\notag\\
 &\le C_T\rho^{1-2\alpha}.                       \label{eq:B-annulus-full}
\end{align}
Here the restriction $\alpha<1/2$ enters through the integrability of $|a|^{-2\alpha}$ at the origin.

Fix $R\ge2R_0$.  For $Rr_h<\varepsilon$, choose $J_h$ so that
\[
 2^{-J_h-1}\varepsilon<Rr_h\le2^{-J_h}\varepsilon.
\]
The region $\{Rr_h\le|x|<\varepsilon\}$ is covered by the annuli
\[
 2^{-j-1}\varepsilon\le|x|<2^{-j}\varepsilon,
 \qquad 0\le j\le J_h.
\]
For the smallest of these annuli the lower radius is larger than
$Rr_h/2\ge R_0r_h$, so \eqref{eq:B-annulus-full} applies to every term.  Therefore \eqref{eq:B-annulus-full} and the geometric series give
\begin{align}
 \E\left[h\sum_{k<N_h(T)}\1_{\{Rr_h\le|Y_k^h|<\varepsilon\}}\right]
 &\le C_T\sum_{j=0}^{J_h}(2^{-j}\varepsilon)^{1-2\alpha}\notag\\
 &\le C_T\varepsilon^{1-2\alpha}.                 \label{eq:B-meso}
\end{align}
Indeed,
\[
 \sum_{j=0}^{J_h}2^{-j(1-2\alpha)}
 \le \sum_{j\ge0}2^{-j(1-2\alpha)}
 =\frac{1}{1-2^{-(1-2\alpha)}}<\infty,
\]
since $\alpha<1/2$.

\subsection{Excursions into the microscopic core}
The following estimate is the microscopic counterpart of the annular bound \eqref{eq:B-meso}.  We state it separately because it is the step that controls repeated returns to the degeneracy.

\begin{lemma}[Microscopic core occupation]\label{lem:core-occupation}
For every fixed $R>0$ and $T<\infty$,
\[
 \E\left[h\sum_{k<N_h(T)}\1_{\{|Y_k^h|<Rr_h\}}\right]
 \le C_R(h+\Theta_h)+C_{T,R}\frac h{r_h}+C_{T,R}e^{-c_R/h}.
\]
In particular, if $\Theta_h\to0$, the expected time spent in the microscopic core tends to zero.
\end{lemma}

\begin{proof}
Let $V_k^h=Y_k^h/r_h$ and
\[
 G_R(v)=\E_v\tau_R,
 \qquad
 \tau_R=\inf\{k\ge0:|V_k|\ge R\}.
\]
Appendix \ref{app:exit} gives
\begin{equation}\label{eq:B-Green}
 G_R(v)\le C_R\{1+H(v)\},
 \qquad 0<|v|<R.
\end{equation}
The initial core excursion therefore contributes at most
\begin{equation}\label{eq:B-initial-exc}
 h\1_{\{0<|v_h|<R\}}G_R(v_h)
 \le C_R(h+\Theta_h).
\end{equation}

For later excursions define, for $|u|\ge R$,
\[
 \mathcal E_R(u)=\E_u[G_R(V_1);|V_1|<R].
\]
To see explicitly how this quantity enters, define successive entrance and exit times by
\[
 \sigma_1=\inf\{k\ge0:|V_k^h|<R\},
 \qquad
 \tau_1=\inf\{k\ge\sigma_1:|V_k^h|\ge R\},
\]
and recursively
\[
 \sigma_{m+1}=\inf\{k\ge\tau_m:|V_k^h|<R\},
 \qquad
 \tau_{m+1}=\inf\{k\ge\sigma_{m+1}:|V_k^h|\ge R\}.
\]
(As usual, the infimum of the empty set is $\infty$.)  We do not condition on $\mathcal F_{\sigma_m-1}$, because $\sigma_m-1$ need not be a stopping time.  Instead, charge every later core excursion to the deterministic grid step at which the chain enters the core.  For each deterministic $k$, the Markov property at time $k$ gives
\[
 \E\!\left[
   \1_{\{|V_k^h|\ge R,\ |V_{k+1}^h|<R\}}
   G_R(V_{k+1}^h)\,\middle|\,\mathcal F_k
 \right]
 =\1_{\{|V_k^h|\ge R\}}\mathcal E_R(V_k^h).
\]
Once an entrance occurs, the expected number of subsequent core steps before the next exit is $G_R(V_{k+1}^h)$; truncating that excursion at $N_h(T)$ can only decrease its contribution.  Summing the preceding deterministic-time identity over $k<N_h(T)$ and separating a possible initial core excursion therefore yields
\begin{equation}\label{eq:B-exc-decomp}
 \E\sum_{k<N_h(T)}\1_{\{|V_k^h|<R\}}
 \le
 \1_{\{0<|v_h|<R\}}G_R(v_h)
 +\E\sum_{k<N_h(T)}\1_{\{|V_k^h|\ge R\}}\mathcal E_R(V_k^h).
\end{equation}
Thus no conditioning at a non-stopping time is used in the excursion decomposition.

\subsection{Returns from a bounded collar}
Fix $M>\max\{2R,R_0\}$.  At scale one write
\[
 \mathcal R(u,a)=|u|^\alpha
 g\left(\frac{|u-a|}{|u|^\alpha}\right).
\]
For $R\le|u|\le M$,
\[
 \mathcal E_R(u)=\int_{-R}^{R}G_R(v)p(u,v)\,dv,
\]
where $p(u,v)$ is the Gaussian transition density.  The bound \eqref{eq:B-Green} shows that $G_R$ is Lebesgue integrable on $(-R,R)$, and the family $p(u,\cdot)$ is uniformly bounded there when $R\le|u|\le M$.  Dominated convergence therefore makes $u\mapsto\mathcal E_R(u)$ continuous on the compact collar.

The function
\[
 u\mapsto\mathcal R(u,R)+\mathcal R(u,-R)
\]
is continuous and strictly positive on the compact collar $\{u:R\le|u|\le M\}$.  Consequently
\begin{equation}\label{eq:B-collar-ratio}
 \mathcal E_R(u)
 \le C_{R,M}\{\mathcal R(u,R)+\mathcal R(u,-R)\},
 \qquad R\le|u|\le M.
\end{equation}
Since $r_h^\alpha\sqrt h=r_h$,
\[
 \mathcal R_h(r_hu,\pm Rr_h)=r_h\mathcal R(u,\pm R).
\]
Multiplying \eqref{eq:B-collar-ratio} by $h$, summing over the collar, and using \eqref{eq:B-telescope-bound} with $a=\pm Rr_h$ yields
\begin{equation}\label{eq:B-collar}
 \E\left[h\sum_{k<N_h(T)}
 \1_{\{R\le|V_k^h|\le M\}}\mathcal E_R(V_k^h)\right]
 \le C_{T,R,M}\frac h{r_h}.
\end{equation}

\subsection{Returns from the far field}
If $|u|\ge2R$, entrance into $(-R,R)$ in one step requires a Gaussian displacement of order $|u|$.  Since the one-step standard deviation is $|u|^\alpha$, the Gaussian tail and \eqref{eq:B-Green} give
\begin{equation}\label{eq:B-far-kernel}
 \mathcal E_R(u)
 \le C_R\exp\{-c_R|u|^{2(1-\alpha)}\}.
\end{equation}
For completeness, this follows from
\[
 \mathcal E_R(u)
 =\int_{-R}^{R}G_R(v)
 \frac{1}{\sqrt{2\pi}|u|^\alpha}
 \exp\left\{-\frac{(v-u)^2}{2|u|^{2\alpha}}\right\}dv
\]
and $|v-u|\ge|u|/2$ for $|u|\ge2R$, with the polynomial prefactor absorbed into the exponential after changing the constants.

Fix $0<\delta\le1$.  For the intermediate region $Mr_h<|Y_k^h|<\delta$, let $J_h$ be the largest integer for which $2^{J_h}Mr_h<\delta$ and decompose into the nonempty intersections
\[
 A_j^h=\{2^jMr_h\le|Y_k^h|<\min(2^{j+1}Mr_h,\delta)\},
 \qquad 0\le j\le J_h.
\]
For every such $j$, the annular radius $2^jMr_h$ is at most $\delta\le1$, so the range restriction in \eqref{eq:B-annulus-full} is respected.  On $A_j^h$, \eqref{eq:B-far-kernel} gives
\[
 \mathcal E_R(V_k^h)
 \le C_R e^{-c_R(2^jM)^{2(1-\alpha)}}.
\]
Applying \eqref{eq:B-annulus-full} at radius $2^jMr_h$ and summing,
\begin{align}
 &\E\left[h\sum_{k<N_h(T)}
 \1_{\{Mr_h<|Y_k^h|<\delta\}}\mathcal E_R(V_k^h)\right]\notag\\
 &\quad\le C_T r_h^{1-2\alpha}
 \sum_{j=0}^{J_h}(2^jM)^{1-2\alpha}
 e^{-c_R(2^jM)^{2(1-\alpha)}}\notag\\
 &\quad\le C_{T,R,M}\frac h{r_h},                 \label{eq:B-intermediate}
\end{align}
because $r_h^{1-2\alpha}=h/r_h$ and the displayed series is finite.

Finally, if $|Y_k^h|\ge\delta$, then $|V_k^h|\ge\delta/r_h$ and \eqref{eq:B-far-kernel} yields
\[
 \mathcal E_R(V_k^h)
 \le C_{R,\delta}\exp\{-c_{R,\delta}/h\},
\]
because $r_h^{-2(1-\alpha)}=h^{-1}$.  Since there are at most $T/h+1$ grid points,
\begin{equation}\label{eq:B-macro}
 \E\left[h\sum_{k<N_h(T)}
 \1_{\{|Y_k^h|\ge\delta\}}\mathcal E_R(V_k^h)\right]
 \le C_{T,R,\delta}e^{-c_{R,\delta}/h}.
\end{equation}

Combining \eqref{eq:B-exc-decomp}, \eqref{eq:B-initial-exc}, \eqref{eq:B-collar}, \eqref{eq:B-intermediate}, and \eqref{eq:B-macro}, and absorbing fixed $M,\delta$ into the constants, gives
\begin{equation}\label{eq:B-core-full}
 \E\left[h\sum_{k<N_h(T)}\1_{\{|Y_k^h|<Rr_h\}}\right]
 \le C_R(h+\Theta_h)+C_{T,R}\frac h{r_h}+C_{T,R}e^{-c_R/h}.
\end{equation}
This proves Lemma \ref{lem:core-occupation}.
\end{proof}

\subsection{From grid occupation to continuous occupation}
Equations \eqref{eq:B-meso} and \eqref{eq:B-core-full} imply, in the subcritical regime $\Theta_h\to0$,
\[
 \limsup_{h\downarrow0}
 \E\left[h\sum_{k<N_h(T)}\1_{\{|Y_k^h|<\varepsilon\}}\right]
 \le C_T\varepsilon^{1-2\alpha}.
\]
Let
\[
 D_h(T)=\sup_{t\le T}|X_t^h-X_{\eta_h(t)}^h|.
\]
For every $\delta>0$,
\[
 \{|X_t^h|<\varepsilon\}\cap\{D_h(T)\le\delta\}
 \subset\{|X_{\eta_h(t)}^h|<\varepsilon+\delta\}.
\]
Therefore
\begin{align*}
 \int_0^T\1_{\{|X_t^h|<\varepsilon\}}dt
 &\le h\sum_{k=0}^{N_h(T)}
 \1_{\{|Y_k^h|<\varepsilon+\delta\}}
 +T\1_{\{D_h(T)>\delta\}}+h.
\end{align*}
Appendix \ref{app:identification} proves $D_h(T)\to0$ in probability.  Taking expectations, then $\limsup_{h\downarrow0}$, and finally $\delta\downarrow0$ gives
\[
 \limsup_{h\downarrow0}
 \E\int_0^T\1_{\{|X_t^h|<\varepsilon\}}dt
 \le C_T\varepsilon^{1-2\alpha},
\]
which is Theorem \ref{thm:occupation}.

\section{Tightness and identification of the non-sticky limit}\label{app:identification}

\subsection{Moment and mesh estimates}
For every $p\ge2$ and $T<\infty$, BDG, H\"older's inequality, and
$|x|^{\alpha p}\le1+|x|^p$ give
\[
 \begin{aligned}
 \E\sup_{u\le t}|X_u^h|^p
 &\le C_p|x_h|^p
 +C_p\E\left(\int_0^t|X_{\eta_h(s)}^h|^{2\alpha}ds\right)^{p/2}\\
 &\le C_{p,T}
 +C_{p,T}\int_0^t\E\sup_{r\le s}|X_r^h|^pds.
 \end{aligned}
\]
For each fixed $h>0$ and $T<\infty$, the left-hand side is finite; this
follows, for instance, by induction over the finitely many Euler steps on
$[0,T]$.  Since $x_h\to0$, the initializations are bounded for all
sufficiently small $h$.  Hence Gronwall's lemma applies and yields, after
restricting to such $h$,
\[
 \sup_{0<h\le h_0}\E\sup_{t\le T}|X_t^h|^p\le C_{p,T}
\]
for some $h_0>0$.
For $0\le s<t\le T$, BDG and H\"older's inequality give
\[
 \begin{aligned}
 \E|X_t^h-X_s^h|^p
 &\le C_p\E\left(\int_s^t|X_{\eta_h(u)}^h|^{2\alpha}du\right)^{p/2}\\
 &\le C_{p,T}|t-s|^{p/2-1}
       \int_s^t\E\bigl[1+|X_{\eta_h(u)}^h|^p\bigr]du\\
 &\le C_{p,T}|t-s|^{p/2}.
 \end{aligned}
\]
Thus, choosing $p>2$, the Kolmogorov--Chentsov criterion gives tightness of
$\{X^h\}$ in $C([0,T])$.  This also supplies the increment estimate quoted
in Section~6 rather than leaving tightness as a separate implicit step.

On $[kh,(k+1)h)$,
\[
 X_t^h-X_{kh}^h=|Y_k^h|^\alpha(W_t-W_{kh}).
\]
Thus, with $D_h(T)=\sup_{t\le T}|X_t^h-X_{\eta_h(t)}^h|$,
\[
 \begin{aligned}
 \E[D_h(T)^p]
 &\le \sum_{k\le N_h(T)}
 \E\left[|Y_k^h|^{\alpha p}
 \sup_{kh\le t\le(k+1)h}|W_t-W_{kh}|^p\right]\\
 &\le C_ph^{p/2}\sum_{k\le N_h(T)}\E|Y_k^h|^{\alpha p}
 \le C_{p,T}h^{p/2-1}.
 \end{aligned}
\]
Choose any $p>2$.  Since $2\alpha<p$,
\[
 \E[D_h(T)^{2\alpha}]
 \le(\E[D_h(T)^p])^{2\alpha/p}\to0.
\]
Because $0<2\alpha<1$,
\[
 \big||x|^{2\alpha}-|y|^{2\alpha}\big|\le|x-y|^{2\alpha},
\]
and hence
\begin{equation}\label{eq:lagremoval-app}
 \E\int_0^T
 \big||X_{\eta_h(s)}^h|^{2\alpha}-|X_s^h|^{2\alpha}\big|ds
 \le T\E[D_h(T)^{2\alpha}]\to0.
\end{equation}

\subsection{Passage to the martingale problem}
Let $h_n\downarrow0$ and suppose $X^{h_n}\Rightarrow X$ in $C([0,T])$.  For
$f\in C_c^2(\R)$,
\[
 M_t^{h,f}=f(X_t^h)-f(x_h)
 -\frac12\int_0^tf''(X_s^h)|X_{\eta_h(s)}^h|^{2\alpha}ds
\]
is a martingale.  Fix
$0\le t_1<\cdots<t_m\le s<t\le T$ and a bounded continuous
$\Phi:\R^m\to\R$.  Then
\begin{equation}\label{eq:cylinder-app}
 \E\left[\Phi(X_{t_1}^h,\ldots,X_{t_m}^h)
 (M_t^{h,f}-M_s^{h,f})\right]=0.
\end{equation}
By \eqref{eq:lagremoval-app}, replacing the lagged coefficient in
\eqref{eq:cylinder-app} by $|X_u^h|^{2\alpha}$ changes its left-hand side by a quantity tending to zero.  Along a Skorokhod representation of the convergent subsequence, $X^{h_n}\to X$ uniformly almost surely.  Therefore
\[
 \int_s^t f''(X_u^{h_n})|X_u^{h_n}|^{2\alpha}du
 \longrightarrow
 \int_s^t f''(X_u)|X_u|^{2\alpha}du
\]
almost surely.  The moment bound above supplies uniform integrability, so passage to the limit in \eqref{eq:cylinder-app} gives
\[
 \E\left[\Phi(X_{t_1},\ldots,X_{t_m})(M_t^f-M_s^f)\right]=0,
\]
where
\[
 M_t^f=f(X_t)-f(0)-\frac12\int_0^tf''(X_u)|X_u|^{2\alpha}du.
\]
Hence $M^f$ is a martingale for every $f\in C_c^2(\R)$.  Localizing with test functions equal to $x$ and $x^2$ on successively larger compact intervals shows that $X$ is a continuous local martingale with
\[
 \langle X\rangle_t=\int_0^t|X_s|^{2\alpha}ds.
\]
After an enlargement of the probability space if necessary, the usual martingale representation therefore gives a Brownian motion $W$ such that
\[
 X_t=\int_0^t|X_s|^\alpha dW_s.
\]

\subsection{Transfer of the occupation estimate}
In the subcritical regime, choose a continuous $\psi_\varepsilon:\R\to[0,1]$ satisfying
\[
 \1_{\{|x|\le\varepsilon/2\}}\le\psi_\varepsilon(x)
 \le\1_{\{|x|<\varepsilon\}}.
\]
The map $f\mapsto\int_0^T\psi_\varepsilon(f(t))dt$ is bounded and continuous on $C([0,T])$.  Theorem \ref{thm:occupation} therefore gives
\[
 \E\int_0^T\psi_\varepsilon(X_t)dt
 =\lim_{n\to\infty}\E\int_0^T\psi_\varepsilon(X_t^{h_n})dt
 \le C_T\varepsilon^{1-2\alpha}.
\]
Since $\1_{\{X_t=0\}}\le\psi_\varepsilon(X_t)$, letting $\varepsilon\downarrow0$ yields
\begin{equation}\label{eq:nonsticky-limit-app}
 \int_0^T\1_{\{X_t=0\}}dt=0\qquad\text{a.s.}
\end{equation}

\subsection{Brownian time change}
Let $B$ be Brownian motion and define
\[
 A_u=\int_0^u\1_{\{B_s\ne0\}}|B_s|^{-2\alpha}ds.
\]
By the Brownian occupation-density formula,
\[
 A_u=\int_\R |x|^{-2\alpha}L_u^x(B)dx.
\]
Because $2\alpha<1$, the singularity $|x|^{-2\alpha}$ is locally integrable; since $x\mapsto L_u^x(B)$ is continuous with compact support for fixed $u$, $A_u<\infty$ almost surely for every finite $u$.  Moreover $A$ is continuous and strictly increasing, because Brownian motion spends zero Lebesgue time at a single point.

We also need $A_u\to\infty$.  By recurrence of one-dimensional Brownian motion, its occupation time of every nonempty bounded interval is infinite almost surely; equivalently, by the occupation-density formula,
\[
 \int_0^\infty \1_{\{|B_s|\le1\}}\,ds
 =\int_{-1}^{1}L_\infty^x(B)\,dx=\infty
 \qquad\text{a.s.}
\]
(One may equally deduce this from the standard divergence of Brownian local time at recurrent points.)  Since $|B_s|^{-2\alpha}\ge1$ on $\{0<|B_s|\le1\}$ and the Brownian zero set has Lebesgue measure zero,
\[
 A_u\ge \int_0^u\1_{\{0<|B_s|\le1\}}\,ds,
\]
so
\begin{equation}\label{eq:Ainfty-app}
 A_u\longrightarrow\infty\qquad\text{a.s.}
\end{equation}

Let
\[
 T_t=\inf\{u\ge0:A_u>t\},\qquad X_t^{\rm tc}=B_{T_t}.
\]
For a nonnegative Borel function $g$, the change of variables $t=A_u$ gives
\[
 \int_0^t g(T_s)ds
 =\int_0^{T_t}g(u)dA_u
 =\int_0^{T_t}g(u)\1_{\{B_u\ne0\}}|B_u|^{-2\alpha}du.
\]
Taking $g(u)=|B_u|^{2\alpha}$, with value zero at $B_u=0$, yields
\begin{equation}\label{eq:inverseclock-app}
 T_t=\int_0^t|B_{T_s}|^{2\alpha}ds.
\end{equation}
Thus $X^{\rm tc}$ is a continuous local martingale with quadratic variation $T_t$, and hence is a weak solution of the SDE.  The same change of variables gives
\[
 \int_0^t\1_{\{X_s^{\rm tc}=0\}}ds
 =\int_0^{T_t}\1_{\{B_u=0\}}dA_u=0,
\]
so this solution is non-sticky.

\subsection{Uniqueness of the non-sticky law}
Conversely, let $X$ be any non-sticky weak solution and put
\[
 Q_t=\int_0^t|X_s|^{2\alpha}ds.
\]
By the non-sticky assumption, $Q$ is strictly increasing: if $Q_t=Q_s$ for $s<t$, then $X_u=0$ for Lebesgue-a.e. $u\in[s,t]$, contradicting non-stickiness.  The Dambis--Dubins--Schwarz theorem \citep{Dambis1965,DubinsSchwarz1965} gives a Brownian motion $B$ such that
\[
 X_t=B_{Q_t}
\]
on the stochastic interval $[0,Q_\infty)$.

We first prove $Q_\infty=\infty$.  Suppose instead that $Q_\infty=q<\infty$.  Extend $B$ continuously to time $q$ (and beyond, if necessary).  Since $Q_t\uparrow q$ and Brownian paths are continuous, $X_t=B_{Q_t}\to B_q$.  If $B_q\ne0$, there are $c>0$ and $t_0$ such that $|X_t|\ge c$ for $t\ge t_0$, and consequently
\[
 Q_\infty-Q_{t_0}=\int_{t_0}^\infty|X_s|^{2\alpha}ds=\infty,
\]
a contradiction.  Hence $B_q=0$.

For $u<q$, let
\[
 S_u=\inf\{t\ge0:Q_t>u\}.
\]
Since $Q$ is continuous and strictly increasing, $S_u<\infty$ and $S_u\uparrow\infty$ as $u\uparrow q$.  Applying the same change-of-variables calculation as above to $Q$ and using $X_{S_u}=B_u$ gives
\[
 S_u=\int_0^u\1_{\{B_s\ne0\}}|B_s|^{-2\alpha}ds=A_u,
 \qquad u<q.
\]
But $A_q<\infty$ by local integrability, whereas $S_u\uparrow\infty$ as $u\uparrow q$, a contradiction.  Therefore
\[
 Q_\infty=\infty\qquad\text{a.s.}
\]
Since $Q_\infty=\infty$, its inverse $S_u=\inf\{t\ge0:Q_t>u\}$ is defined for every $u\ge0$.  The preceding identity gives $S_u=A_u$.  Taking inverses, $Q_t=T_t$ for every $t\ge0$, where $T$ is the inverse of $A$, and hence
\[
 X_t=B_{T_t}.
\]
Thus the Brownian time-change construction uniquely determines the law of a non-sticky weak solution.  Together with the martingale-problem identification and the occupation estimate, this completes the proof of Theorem \ref{thm:subcritical}.

\section*{Acknowledgments}
The author used OpenAI ChatGPT (GPT-5.6 Sol, 2026) as an assistive tool for language editing, structural suggestions, literature searches, numerical-code review, exploration of alternative formulations and proof strategies, and critical examination of mathematical arguments. All mathematical arguments, numerical results, and references were independently evaluated by the author; all substantive revisions and decisions regarding inclusion were made by the author. The author assumes full responsibility for the originality, correctness, and integrity of the manuscript.

\par\medskip
\noindent The views expressed in this article are solely those of the author and do not necessarily reflect those of Sumitomo Mitsui Banking Corporation.


\begin{thebibliography}{99}
\bibitem[Antonov et~al.(2015)Antonov, Konikov and Spector]{AntonovKonikovSpector2015}
Antonov, A., Konikov, M. and Spector, M. (2015). The free boundary SABR: natural extension to negative rates. \emph{Risk}, September, 68--73.

\bibitem[Bafico and Baldi(1982)]{BaficoBaldi1982}
Bafico, R. and Baldi, P. (1982). Small random perturbations of Peano phenomena. \emph{Stochastics} \textbf{6}, 279--292. doi:10.1080/17442508208833208.

\bibitem[Bass(2014)]{Bass2014}
Bass, R.F. (2014). A stochastic differential equation with a sticky point. \emph{Electron. J. Probab.} \textbf{19}, no. 32, 1--22. doi:10.1214/EJP.v19-2350.

\bibitem[Bass et~al.(2007)Bass, Burdzy and Chen]{BassBurdzyChen2007}
Bass, R.F., Burdzy, K. and Chen, Z.-Q. (2007). Pathwise uniqueness for a degenerate stochastic differential equation. \emph{Ann. Probab.} \textbf{35}, 2385--2418. doi:10.1214/009117907000000033.

\bibitem[Bou-Rabee and Holmes-Cerfon(2020)]{BouRabeeHolmesCerfon2020}
Bou-Rabee, N. and Holmes-Cerfon, M.C. (2020). Sticky Brownian motion and its numerical solution. \emph{SIAM Rev.} \textbf{62}, 164--195. doi:10.1137/19M1268446.

\bibitem[Dambis(1965)]{Dambis1965}
Dambis, K.E. (1965). On the decomposition of continuous submartingales. \emph{Theory Probab. Appl.} \textbf{10}, 401--410. doi:10.1137/1110048.

\bibitem[Delarue and Flandoli(2014)]{DelarueFlandoli2014}
Delarue, F. and Flandoli, F. (2014). The transition point in the zero noise limit for a 1D Peano example. \emph{Discrete Contin. Dyn. Syst.} \textbf{34}, 4071--4083. doi:10.3934/dcds.2014.34.4071.

\bibitem[Dubins and Schwarz(1965)]{DubinsSchwarz1965}
Dubins, L.E. and Schwarz, G. (1965). On continuous martingales. \emph{Proc. Natl. Acad. Sci. USA} \textbf{53}, 913--916. doi:10.1073/pnas.53.5.913.

\bibitem[Engelbert and Peskir(2014)]{EngelbertPeskir2014}
Engelbert, H.-J. and Peskir, G. (2014). Stochastic differential equations for sticky Brownian motion. \emph{Stochastics} \textbf{86}, 993--1021. doi:10.1080/17442508.2014.899600.

\bibitem[Engelbert and Schmidt(1985)]{EngelbertSchmidt1985}
Engelbert, H.-J. and Schmidt, W. (1985). On solutions of one-dimensional stochastic differential equations without drift. \emph{Z. Wahrsch. Verw. Gebiete} \textbf{68}, 287--314. doi:10.1007/BF00532642.

\bibitem[Gy\"ongy and R\'asonyi(2011)]{GyongyRasonyi2011}
Gy\"ongy, I. and R\'asonyi, M. (2011). A note on Euler approximations for SDEs with H\"older continuous diffusion coefficients. \emph{Stochastic Process. Appl.} \textbf{121}, 2189--2200. doi:10.1016/j.spa.2011.06.008.

\bibitem[Peskir(2022)]{Peskir2022}
Peskir, G. (2022). Sticky Bessel diffusions. \emph{Stochastic Process. Appl.} \textbf{150}, 1015--1036. doi:10.1016/j.spa.2020.12.009.

\bibitem[Ram\'irez and Uribe Bravo(2024)]{RamirezUribe2024}
Ram\'irez, M. and Uribe Bravo, G. (2024). The sticky L\'evy process as a solution to a time change equation. \emph{J. Math. Anal. Appl.} \textbf{530}, 127742. doi:10.1016/j.jmaa.2023.127742.

\bibitem[Sharma(2025)]{Sharma2025}
Sharma, A. (2025). Weak approximation of stochastic differential equations with sticky boundary conditions. arXiv:2508.06487.

\bibitem[Suganuma and Yamada(2017)]{SuganumaYamada2017}
Suganuma, K. and Yamada, T. (2017). Forward-rate models allowing for negative interest rates and market expectations of future interest rates. IMES Discussion Paper Series 2017-J-18, Institute for Monetary and Economic Studies, Bank of Japan. [In Japanese.]

\bibitem[Taguchi and Tanaka(2019)]{TaguchiTanaka2019}
Taguchi, D. and Tanaka, A. (2019). On the Euler--Maruyama scheme for degenerate stochastic differential equations with non-sticky condition. In: \emph{S\'eminaire de Probabilit\'es L}, Lecture Notes in Mathematics 2252, 165--185. Springer. doi:10.1007/978-3-030-28535-7\_9.

\bibitem[Yan(2002)]{Yan2002}
Yan, L. (2002). The Euler scheme with irregular coefficients. \emph{Ann. Probab.} \textbf{30}, 1172--1194. doi:10.1214/aop/1029867124.
\end{thebibliography}
\end{document}